\documentclass[12pt]{amsart}

\usepackage{amsmath,amsthm}

\usepackage{amssymb}

\usepackage{pdfsync}

\newcommand{\R}{{\mathbb R}}       
\newcommand{\Z}{{\mathbb Z}}       
\newcommand{\DD}{{\mathcal D}}
\newcommand{\HH}{{\mathcal H}}

\newcommand{\BZ}{{\mathcal B}}
\newcommand{\GZ}{{\mathcal G}}

\newcommand{\WW}{{\mathcal W}}

\newcommand{\diam}{{\rm diam}}
\newcommand{\dist}{{\rm dist}}
\newcommand{\ds}{\displaystyle }

\newcommand{\fiproof}{{\hspace*{\fill} $\square$ \vspace{2pt}}}

\newcommand{\rf}[1]{{\eqref{#1}}}
\newcommand{\supp}{\operatorname{supp}}

\newcommand{\vphi}{{\varphi}}
\newcommand{\ve}{{\varepsilon}}
\newcommand{\vv}{{\vspace{2mm}}}
\newcommand{\vvv}{{\vspace{3mm}}}
\newcommand{\wt}[1]{{\widetilde{#1}}}

\newcommand{\sss}{{\rm Stop}}

\newcommand{\GG}{{\mathcal G}}
\newcommand{\LD}{{\mathsf{LD}}}
\newcommand{\HD}{{\mathsf{HD}}}

\newcommand{\BA}{{\mathsf{BA}}}

\newtheorem{theorem}{Theorem}[section]
\newtheorem*{theorem*}{Theorem}
\newtheorem*{theorema*}{Theorem A}
\newtheorem*{theoremb*}{Theorem B}
\newtheorem*{theoremc*}{Theorem C}
\newtheorem*{theoremd*}{Theorem D}
\newtheorem{lemma}[theorem]{Lemma}

\newtheorem{coro}[theorem]{Corollary}

\theoremstyle{definition}

\theoremstyle{remark}

\numberwithin{equation}{section}

\begin{document}

\begin{abstract}
In this paper it is shown that if $\mu$ is a finite Radon measure in $\R^d$ which is $n$-rectifiable and $1\leq p\leq 2$,
then 
$$\int_0^\infty \beta_{\mu,p}^n(x,r)^2\,\frac{dr}r<\infty \quad\mbox{ for $\mu$-a.e. $x\in\R^d$,}$$
where 
$$\beta_{\mu,p}^n(x,r) = \inf_L \left(\frac1{r^n} \int_{\bar B(x,r)} \left(\frac{\dist(y,L)}{r}\right)^p\,d\mu(y)\right)^{1/p},$$
with the infimum taken over all the $n$-planes $L\subset \R^d$. The $\beta_{\mu,p}^n$ coefficients 
are the same as the ones considered by David and Semmes in the setting of the so called 
uniform $n$-rectifiability.
An analogous necessary condition
for $n$-rectifiability in terms
of other coefficients involving some variant of the Wasserstein distance $W_1$ is also proved.
\end{abstract}

\title[Rectifiability in terms of Jones' square function: Part I]{Characterization of 
$n$-rectifiability in terms of Jones' square function: Part~I}


\author{Xavier Tolsa}
\address{Xavier Tolsa, ICREA and Universitat Aut\`onoma de Barcelona, Barcelona, Catalonia}

\thanks{The author was supported by 
the ERC grant 320501 of the European Research
Council (FP7/2007-2013) 
and also partially supported by the
grants
2014-SGR-75 (Catalonia), MTM2013-44304-P (Spain),  and by the Marie Curie ITN MAnET (FP7-607647). 
}

\maketitle








\section{Introduction}

A set $E\subset \R^d$ is called $n$-rectifiable if there are Lipschitz maps
$f_i:\R^n\to\R^d$, $i=1,2,\ldots$, such that 
$$
\HH^n\biggl(\R^d\setminus\bigcup_i f_i(\R^n)\biggr) = 0,
$$
where $\HH^n$ stands for the $n$-dimensional Hausdorff measure. On the other hand, one says that a 
Radon measure $\mu$ on $\R^d$ is $n$-rectifiable if $\mu$ vanishes out of an $n$-rectifiable
set $E\subset\R^d$ and moreover $\mu$ is absolutely continuous with respect to $\HH^n_E$.

One of the main objectives of geometric measure theory consist in obtaining different characterizations 
of $n$-rectifiability. For example, there are classical characterizations in terms of the existence of approximate
tangents, in terms of the existence of densities, or in terms of the size of orthogonal projections. For the precise
statements and proofs of these nice results the reader is referred to \cite{Mattila-book}.

More recently, the development of quantitative rectifiability in the pioneering works of
Jones \cite{Jones} and David and Semmes \cite{DS1} has led to the study of the connection
between rectifiability and the boundedness of square functions and singular integrals
(for instance, see \cite{David-vitus}, \cite{Leger}, \cite{NToV} or \cite{CGLT}). Many results on this
subject deal with the so called uniform $n$-rectifiability introduced by David and Semmes \cite{DS2}.
One says that $\mu$ is uniformly $n$-rectifiable if 
it is $n$-AD-regular, that is $c^{-1}r^n\leq \mu(B(x,r))\leq c\,r^n$ for all $x\in\supp\mu, \,r>0$ and 
some constant $c>0$,
and further there exist constants $\theta,M>0$ so that, for
each $x\in\supp\mu$ and $R>0$, there is a Lipschitz mapping $g$ from the $n$-dimensional ball $B_n(0,r)\subset\R^n$
to $\R^d$ such that $g$ has Lipschitz norm not exceeding $M$ and
$$\mu\bigl(B(x,r)\cap g(B_n(0,r))\bigr) \geq \theta r^n.$$

To state one of the main result of \cite{DS1} we need to introduce some additional notation.
Given $1<p<\infty$, a closed ball $B \subset \R^d$, and an integer $0<n<d$, let
$$\beta_{\mu,p}^n(B) = \inf_L \left(\frac1{r(B)^n} \int_{B} \left(\frac{\dist(y,L)}{r(B)}\right)^p\,d\mu(y)\right)^{1/p},$$
where the infimum is taken over all the $n$-planes $L\subset \R^d$. Quite often, 
given a fixed $n$, to simplify notation we will drop the exponent $n$ and we will write
$\beta_{\mu,p}(x,r)$ instead of $\beta_{\mu,p}^n(\bar B(x,r))$.
The aforementioned result from \cite{DS1} is the following.

\begin{theorema*}
Let $1\leq p < 2n/(n-2)$.
Let $\mu$ be an $n$-AD-regular Borel measure on $\R^d$. The measure $\mu$ is uniformly $n$-rectifiable
if and only if there exists some constant $c>0$ such that
$$\int_{B(x,r)}\int_0^r \!\beta_{\mu,p}^n(y,r)^2 \,\frac{dr}r\,d\mu(y)\leq c\,r^n
\quad \mbox{\,for all $x\in\supp\mu$ and all $r>0$.}$$
\end{theorema*}

In the case $n=1$, a result analogous to this one in terms of $L^\infty$ versions of the
coefficients $\beta_{\mu,p}$ is also valid, even without the $n$-AD-regularity assumption on $\mu$, as shown previously by Jones \cite{Jones} 
in his traveling salesman theorem in the plane, and by Jones \cite{Jones} and Okikiolu \cite{Oki} in $\R^d$.

Other coefficients which involve a variant of the Wasserstein distance $W_1$ in the spirit of the $\beta_{\mu,p}$'s have been introduced in \cite{Tolsa-plms}
and have shown to be useful in the study of different questions regarding the connection between 
uniform $n$-rectifiability and the boundedness of $n$-dimensional singular integral operators
(see \cite{Tolsa-jfa} or \cite{Mas-Tolsa}, for example). 
Given two finite Borel measures $\sigma$, $\mu$
on $\R^d$ and a closed ball $B\subset\R^d$, we set
$$\dist_B(\sigma,\mu):= \sup\Bigl\{ \Bigl|{\textstyle \int f\,d\sigma  -
\int f\,d\mu}\Bigr|:\,{\rm Lip}(f) \leq1,\,\supp(f)\subset
B\Bigr\},$$
where ${\rm Lip}(f)$ stands for the Lipschitz constant of $f$.
We also set
$$
\alpha_\mu^n(B)= \frac1{r(B)^{n+1}}\,\inf_{a\geq0,L} \,\dist_{3B}(\mu,\,a\HH^n_{L}),$$
where $\HH^n_{L}$ stands for the restricion of $\HH^n$ to $L$ and the infimum is taken over all the constants $a\geq0$ and all the $n$-planes $L$ which intersect $B$. 
Again we will drop the exponent $n$ and we will write $\alpha_\mu(x,r)$ instead of
$\alpha_\mu^n(\bar B(x,r))$ to simplify the notation.

In \cite{Tolsa-plms} the following is proved:

\begin{theoremb*}
Let $\mu$ be an $n$-AD-regular Borel measure on $\R^d$. The measure $\mu$ is uniformly $n$-rectifiable
if and only if there exists some constant $c>0$ such that
$$\int_{B(x,r)}\int_0^r \alpha_{\mu}^n(y,r)^2 \,\frac{dr}r\,d\mu(y)\leq c\,r^n
\quad \mbox{\,for all $x\in\supp\mu$ and all $r>0$.}$$
\end{theoremb*}

\vv
In recent years there has been considerable interest in the field of geometric measure theory to obtain appropriate versions of Theorem A and Theorem B
which apply to $n$-rectifiable measures which are not $n$-AD-regular. The need for such results
is specially notorious in the case $n>1$, where there is no analogue of Jones' traveling salesman mentioned above.
The current paper contributes to fill in this gap by means of the following theorem,
which provides necessary conditions for $n$-rectifiability in terms of
the $\beta_{\mu,p}$ coefficients.

\begin{theorem}\label{teo2}
Let $1\leq p\leq 2$. 
Let $\mu$ be a finite Borel measure in $\R^d$ which is $n$-rectifiable. Then
\begin{equation}\label{eqbeta**0}
\int_0^\infty \beta_{\mu,p}^n(x,r)^2\,\frac{dr}r<\infty \quad\mbox{ for $\mu$-a.e. $x\in\R^d$.}
\end{equation}
\end{theorem}

The integral on the left hand side of \rf{eqbeta**0} quite often is called Jones' square function.
In the sequel \cite{AT} of the present work, by Azzam and the author of the present paper, it is shown that the finiteness of Jones' square function for $p=2$ implies 
$n$-rectifiability. The precise result is the following:

\vv

{\em
Let $\mu$ be a finite Borel measure in $\R^d$ such that
\begin{equation}\label{eqlimsup} 
0<\limsup_{r\to 0}\frac{\mu(B(x,r))}{r^n}<\infty\quad\mbox{and}\quad\int_0^\infty \beta_{\mu,2}^n(x,r)^2\,\frac{dr}r<\infty
\end{equation}
for $\mu$-a.e.\ $x\in\R^d$. Then $\mu$ is $n$-rectifiable.}
\vv

So we have:

\begin{coro}[\cite{AT}]
Let $\mu$ be a finite Borel measure in $\R^d$ such that 
$0<\limsup_{r\to 0}\frac{\mu(B(x,r))}{r^n}<\infty$
for $\mu$-a.e.\ $x\in\R^d$. Then $\mu$ is $n$-rectifiable if and only if 
\begin{equation}\label{eqbeta**}
\int_0^\infty \beta_{\mu,2}^n(x,r)^2\,\frac{dr}r<\infty\qquad\mbox{for $\mu$-a.e.\ $x\in\R^d$.}
\end{equation}
In particular, a set $E\subset\R^d$ with $\HH^n(E)<\infty$ is $n$-rectifiable if and only if \eqref{eqbeta**}
holds for $\mu =\HH^n_E$.
\end{coro}

The second result that is obtained in the current paper is the following.

\begin{theorem}\label{teo1}
Let $\mu$ be a finite Borel measure in $\R^d$ which is $n$-rectifiable. Then
$$\int_0^\infty \alpha_\mu^n(x,r)^2\,\frac{dr}r<\infty \quad\mbox{ for $\mu$-a.e.\ $x\in\R^d$.}$$
\end{theorem}

\vv
This theorem can be considered as a version for non-AD-regular measures of Theorem B above.

Let us remark that Theorem \ref{teo2} has already been proved by Pajot \cite{Pajot} under the additional
assumption that $\mu$ is $n$-AD-regular, for $1\leq p < n/(n-2)$. Further, in the same paper he has obtained the following
partial converse:

\begin{theoremc*}
Let $1\leq p< n/(n-2)$. Suppose that $E\subset \R^d$ is compact and that $\mu = \HH^n|E$ is finite. If 
$$\liminf_{r\to 0}\frac{\mu(B(x,r))}{r^n}>0 \quad\mbox{ and }\quad
\int_0^\infty \beta_{\mu,p}^n(x,r)^2\,\frac{dr}r<\infty$$
for $\mu$-a.e. $x\in\R^d$, then $E$ is $n$-rectifiable.
\end{theoremc*}
\vv

Notice that in the above theorem the lower density $\liminf_{r\to 0}\frac{\mu(B(x,r))}{r^n}$ is required to be positive,
while in \rf{eqlimsup} it is the upper density which must be positive. 
Recall that the assumption that the upper density is positive $\mu$-a.e.\ is satisfied for all measures of the form 
$\mu=\HH^n_E$, with $\HH^n(E)<\infty$. On the contrary, the lower density may be zero $\mu$-a.e.\ for this type of measures.  

Quite recently, Badger and Schul \cite{BS2} have shown that Theorem C also holds for 
other measures different from Hausdorff measures, namely for
Radon measures $\mu$ satisfying $\mu\ll\HH^n$. However, their extension of Pajot's theorem still requires the lower density
$\liminf_{r\to 0}\frac{\mu(B(x,r))}{r^n}$ to be positive $\mu$-a.e.

To describe another previous result of Badger and Schul \cite{Badger-Schul} we need to introduce some additional
terminology.
 We say that $\mu$ is $n$-rectifiable in the sense of Federer if there are
Lipschitz maps 
$f_i:\R^n\to\R^d$, $i=1,2,\ldots$, such that 
$$
\mu\biggl(\R^d\setminus\bigcup_i f_i(\R^n)\biggr) = 0.
$$
The condition that $\mu$ is absolutely continuous with respect to $\HH^n$ is not required.

Given a cube $Q\subset \R^d$, denote
$$\wt \beta_{\mu,2}^n(Q) = \inf_L \left(\frac1{\mu(3Q)} \int_{3Q} \left(\frac{\dist(y,L)}{\ell(Q)}\right)^2\,d\mu(y)\right)^{1/2},$$
where $\ell(Q)$ stands for the side length of $Q$ and the infimum is taken over all $n$-planes $L\subset\R^d$.
The result of Badger and Schul in \cite{Badger-Schul} reads as follows:

\begin{theoremd*}
If $\mu$ is a locally finite Borel measure on $\R^d$ which is $1$-rectifiable in the sense of
Federer, then
\begin{equation}\label{conjones}
\sum_{Q\in\DD:x\in Q,\ell(Q)\leq 1}\wt\beta_{\mu,2}^1(Q)^2\,\frac{\ell(Q)}{\mu(Q)}<\infty
\quad\mbox{ for $\mu$-a.e. $x\in\R^d$,}
\end{equation}
where $\DD$ stands for the lattice of dyadic cubes of $\R^d$. 
\end{theoremd*}

According to \cite{Badger-Schul}, Peter Jones conjectured in 2000 that some condition in the spirit of
\rf{conjones} should be necessary and sufficient for $n$-rectifiability (in the sense of Federer).
Observe that from Theorem \ref{teo2} it follows easily that if $\mu$ is $n$-rectifiable (in the sense that $\mu\ll\HH^n$),
then
\begin{equation}\label{cond483}
\sum_{Q\in\DD:x\in Q,\ell(Q)\leq 1}\wt\beta_{\mu,2}^n(Q)^2<\infty
\quad\mbox{ for $\mu$-a.e. $x\in\R^d$.}
\end{equation}


Notice that Theorem D is only proved in the case $n=1$. As remarked by the authors in 
\cite{Badger-Schul}, it is not clear how one could extend their techniques to the
case $n>1$. However, in contrast to Theorem \ref{teo2} their result has the advantage  that it applies to measures that need not be absolutely continuous with respect to $\HH^1$. 

For another work in connection
with rectifiability and other variants of the $\beta_2$ coefficients, we suggest the reader to see Lerman's
work \cite{Lerman}, and for two recent papers which involve some variants of the $\alpha$ coefficients without 
the AD-regularity assumption, see \cite{ADT1} and \cite{ADT2}.

\vv
The plan of the paper is the following. First we prove Theorem \ref{teo1} in Section \ref{sec1}.
We carry out this task by combining suitable stopping time arguments with the application
of Theorem B to the particular case when $\mu$ is $n$-dimensional Hausdorff measure on an $n$-dimensional
Lipschitz graph. Finally, we show in Section \ref{sec2} that Theorem \ref{teo2} follows from Theorem 
\ref{teo1} by means of other stopping time arguments. Both in Theorem \ref{teo2} and \ref{teo1}, the stopping time arguments are mainly used to control the oscillations of the density of $\mu$ at different scales.

\vv
In this paper the letters $c,C$ stand
for some absolute constants which may change their values at different
occurrences. On the other hand, constants with subscripts, such as $c_1$, do not change their values
at different occurrences.
The notation $A\lesssim B$ means that
there is some fixed constant $c$ (usually an absolute constant) such that $A\leq c\,B$. Further, $A\approx B$ is equivalent to $A\lesssim B\lesssim A$. We will
also write $A\lesssim_{c_1} B$ if we want to make explicit 
the dependence on the constants $c_1$ of the relationship ``$\lesssim$''.

\vvv


\section{The proof of Theorem \ref{teo1}}\label{sec1}

\subsection{The Main Lemma}


In this section we will prove the following:

\begin{lemma}[Main Lemma] \label{mainlemma*}
Let $\mu$ be a finite Borel measure on $\R^d$ and let
$\Gamma\subset\R^d$ be an $n$-dimensional Lipschitz graph in $\R^d$. Then
$$\int_0^\infty \alpha_\mu(x,r)^2\,\frac{dr}r<\infty \quad\mbox{ for $\HH^n$-a.e. $x\in\Gamma$.}$$
\end{lemma}
\vv

It is clear that Theorem \ref{teo1} follows as a corollary of the preceding result, taking into account that if $\mu$ is
$n$-rectifiable, then 
it is absolutely continuous with respect to $\HH^n$ restricted to a countable union of (possibly rotated) $n$-dimensional
Lipschitz graphs.

In the remaining of this section we assume that $\mu$ is a finite Borel measure and $\Gamma$ is an $n$-dimensional 
Lipschitz graph, as in Lemma \ref{mainlemma*}.

\vv

\subsection{The case when $\mu$ is supported on $\Gamma$}

We will prove first the following partial result.

\begin{lemma} \label{lem1}
Let
$\Gamma\subset\R^d$ be an $n$-dimensional Lipschitz graph in $\R^d$, and 
let $\mu$ be a finite Borel measure supported on $\Gamma$. Then
$$\int_0^\infty \alpha_\mu(x,r)^2\,\frac{dr}r<\infty \quad\mbox{ for $\HH^n$-a.e. $x\in\Gamma$.}$$
\end{lemma}

To prove the preceding result, we need to introduce the following ``$\Gamma$-cubes'' 
associated with $\Gamma$.
We assume that $\Gamma$ equals the graph of a Lipschitz function $A:\R^n\to\R^{d-n}$.
We say that $Q\subset\Gamma$ is a $\Gamma$-cube if it is a subset of the form $Q=\Gamma\cap (Q_0\times\R^{d-n})$, where
$Q_0\subset \R^n$ is an $n$-dimensional cube. We denote $\ell(Q):=\ell(Q_0)$.
We say that $Q$ is a dyadic $\Gamma$-cube if $Q_0$ is a dyadic
cube. The center of $Q$ is the point $x_Q=(x_{Q_0},A(x_{Q_0}))$, where $x_{Q_0}$ is the center of $Q_0$.
The collection of dyadic $\Gamma$-cubes $Q$ with $\ell(Q)=2^{-j}$ is denoted by $\DD_{\Gamma,j}$. Also,
we set $\DD_{\Gamma}=\bigcup_{j\in\Z}\DD_{\Gamma,j}$ and $\DD_{\Gamma}^k=\bigcup_{j\geq k}\DD_{\Gamma,j}$.

Given a $\Gamma$-cube $Q$, we denote by $B_Q$ a closed ball concentric with $Q$ with $r(B_Q)=3\diam(Q)$. Note that $B_Q$ contains
$Q$ and is centered on $\Gamma$. We set
$$\alpha_\mu(Q) :=\alpha_\mu(B_Q).$$
For a fixed $n$-plane $L$, we also write
$$\alpha_{\mu,L}(Q)= \frac1{r(B_Q)^{n+1}}\,\inf_{a\geq0} \,\dist_{B_Q}(\mu,\,a\HH^n_{L}),$$
so that, if $L$ coincides with the $n$-plane that minimizes $\alpha(Q)$, we have 
$\alpha_{\mu,L}(Q)= \alpha_{\mu}(Q)$.

By standard methods, it follows that Lemma \ref{lem1} is an immediate consequence of the following more precise result:

\begin{lemma} \label{lem1'}
Let
$\Gamma\subset\R^d$ be an $n$-dimensional Lipschitz graph in $\R^d$, 
let $\mu$ be a finite Borel measure supported on $\Gamma$, and let $\DD_\Gamma$ be the lattice of $\Gamma$-cubes introduced above.
Denoting by $L_Q$ the
$n$-plane that minimizes $\alpha_{\HH^n_\Gamma}(Q)$, we have
$$\sum_{Q\in\DD_\Gamma:x\in Q}\alpha_{\mu,L_Q}(Q)^2<\infty
\quad\mbox{ for $\HH^n$-a.e. $x\in\Gamma$.}$$
\end{lemma}

\begin{proof}
First note that if $\mu =g\, \HH^n_\Gamma$ with $g\in L^\infty(\HH^n_\Gamma)$, then the conclusion of the lemma is true.
Indeed, it turns out that the measure
$\sigma := (1 + g)\,\HH^n_\Gamma$
is $n$-AD-regular and thus 
$$\sum_{Q\in\DD_\Gamma:x\in Q,\ell(Q)\leq1}\alpha_{\sigma,L_Q}(Q)^2<\infty
\quad\mbox{ for $\HH^n$-a.e. $x\in\Gamma$,}$$
by the results from \cite{Tolsa-plms}.
Then, for any $Q\in\DD_\Gamma$ we have
$$\alpha_{\mu,L_Q}(Q) \leq \alpha_{\HH^n_\Gamma,L_Q}(Q) + 
\alpha_{\sigma,L_Q}(Q),$$
and thus 
\begin{align*}
\sum_{Q\in\DD_\Gamma:x\in Q,\ell(Q)\leq1}\alpha_{\mu,L_Q}(Q)^2 & \lesssim
\sum_{Q\in\DD_\Gamma:x\in Q,\ell(Q)\leq1}\alpha_{\HH^n_\Gamma,L_Q}(Q)^2\\
&\quad
+\sum_{Q\in\DD_\Gamma:x\in Q,\ell(Q)\leq1}\alpha_{\sigma,L_Q}(Q)^2<\infty
\quad\mbox{ for $\HH^n$-a.e. $x\in\Gamma$,}
\end{align*}
On the other hand, using just that $\mu$ is a finite measure, it follows easily that
$$\sum_{Q\in\DD_\Gamma:x\in Q,\ell(Q)\geq1}\alpha_{\mu,L_Q}(Q)^2 <\infty
\quad\mbox{ for all $x\in\Gamma$,}
$$
and so we have
$$\sum_{Q\in\DD_\Gamma:x\in Q}\alpha_{\mu,L_Q}(Q)^2  <\infty
\quad\mbox{ for $\HH^n$-a.e. $x\in\Gamma$.}$$

\vv

In the case when $\mu$ is an arbitrary finite Borel measure supported on $\Gamma$, we consider a suitable 
Calder\'on-Zygmund decomposition of $\mu$ with respect to $\HH^n_\Gamma$. So for some big $\lambda>0$, we 
let $\BZ$ be the maximal family of $\Gamma$-cubes $P$ such that $\mu(P)>\lambda\,\HH^n(P)$, and then we have
$$\mu = g\,\HH^n_\Gamma + \nu,$$
where
$$g\,\HH^n_\Gamma = \mu|_G + \sum_{P\in \BZ} \frac{\mu(P)}{\HH^n(P)}\,\HH^n|_P,$$
with $G= \R^d \setminus \bigcup_{P\in\BZ} P$, and also
$$\nu= \sum_{P\in\BZ} \nu_P,\qquad
\nu_P = \mu|_P - \frac{\mu(P)}{\HH^n(P)}\,\HH^n|_P.$$
 By standard arguments, it is clear that $\|g\|_{L^\infty(\mu)}\lesssim \lambda$, and thus
if we denote $\sigma= g\,d\HH^n_\Gamma$, then
\begin{equation}\label{eqgg11}
\sum_{Q\in\DD_\Gamma:x\in Q}\alpha_{\sigma,L_Q}(Q)^2<\infty
\quad\mbox{ for $\HH^n$-a.e. $x\in\Gamma$.}
\end{equation}

Let $\wt G= \Gamma\setminus \bigcup_{P\in\BZ} 3B_{P}$, and observe that
$$\HH^n(\Gamma\setminus \wt G) = \sum_{P\in\BZ} \HH^n(3B_{P})\leq c\sum_{P\in\BZ} \HH^n(P)\leq \frac c\lambda \sum_{P\in\BZ} \mu(P)
\leq\frac c\lambda\,\|\mu\|.$$
Since $\lambda$ can be taken arbitrarily big, to prove the lemma it suffices to show that
\begin{equation}\label{eqsuf9*}
\sum_{Q\in\DD_\Gamma:x\in Q}\alpha_{\mu,L_Q}(Q)^2<\infty
\quad\mbox{ for $\HH^n$-a.e. $x\in \wt G$.}
\end{equation}

Denote by $\GG$ the family of those $\Gamma$-cubes which are not contained in $\bigcup_{P\in\BZ} 3B_{P}$.
Notice that, for $x\in \wt G$, all the $\Gamma$-cubes in the sum in \rf{eqsuf9*} are from $\GG$.
For a given $Q\in\GG$, we have
$$\alpha_{\mu,L_Q}(Q) \leq \alpha_{\sigma,L_Q}(Q) + \frac1{\ell(Q)^{n+1}}\,\sup_f \left|\int f\,d\nu\right|,
$$
where the supremum is taken over all functions $f$ which are $1$-Lipschitz and supported on $B_Q$.
To estimate the last integral on the right hand side, we write
$$\left|\int f\,d\nu\right|\leq \sum_{P\in \BZ:P\cap B_Q\neq \varnothing}\left|\int f\,d\nu_P\right|.$$
Since $\int d\nu_P=0$ for all $P\in\BZ$, we have
$$\left|\int f\,d\nu_P\right| \leq \int |f-f(x_P)|\,d|\nu_P|\lesssim \ell(P) \,\|\nu_P\| \lesssim \ell(P)\,\mu(P).$$
Thus,
$$\alpha_{\mu,L_Q}(Q) \lesssim \alpha_{\sigma,L_Q}(Q) + \sum_{P\in \BZ:P\cap B_Q\neq \varnothing} \frac{\ell(P)\,\mu(P)}{\ell(Q)^{n+1}}.$$
So, because of \rf{eqgg11}, to prove the lemma it suffices to show that
\begin{equation}\label{eqsuf99*}
\sum_{Q\in\GG:x\in Q}\biggl(\sum_{P\in \BZ:P\cap B_Q\neq \varnothing} \frac{\ell(P)\,\mu(P)}{\ell(Q)^{n+1}}\biggr)^2<\infty
\quad\mbox{ for $\HH^n$-a.e. $x\in \wt G$.}
\end{equation}

To prove \rf{eqsuf99*}, first we take into account that
 if $P\in\BZ$ is such that $P\cap B_Q\neq
\varnothing$, then
\begin{equation}\label{eqrbp8}
r(B_P)\leq r(B_Q),
\end{equation}
 because otherwise
$Q\subset B_Q\subset 3B_P,$
which contradicts the fact that $Q\in\GG$. Now, from \rf{eqrbp8} we deduce that $P\subset B_P\subset 3B_Q$, and 
taking also into account that $\mu(P)\lesssim
\lambda \,\ell(P)^n$, we get
$$\sum_{P\in \BZ:P\cap B_Q\neq \varnothing} \frac{\ell(P)\,\mu(P)}{\ell(Q)^{n+1}} \lesssim 
\lambda \sum_{P\in \BZ:P\subset 3B_Q} \frac{\ell(P)^{n+1}}{\ell(Q)^{n+1}} \lesssim \lambda.$$
Therefore,
$$\sum_{Q\in\GG:x\in Q}\biggl(\sum_{P\in \BZ:P\cap B_Q\neq \varnothing} \frac{\ell(P)\,\mu(P)}{\ell(Q)^{n+1}}\biggr)^2\lesssim \lambda \sum_{Q\in\GG:x\in Q}\sum_{P\in \BZ:P\cap B_Q\neq \varnothing} \frac{\ell(P)\,\mu(P)}{\ell(Q)^{n+1}}.$$
Now we can show that the last sum is finite for $\HH^n$-a.e.\ $x\in\wt G$ by integrating this with respect to $\HH^n_\Gamma$.
Then we obtain
\begin{align*}
\int_\Gamma \sum_{Q\in\GG:x\in Q}\sum_{P\in \BZ:P\cap B_Q\neq \varnothing} \frac{\ell(P)\,\mu(P)}{\ell(Q)^{n+1}}\,d\HH^n(x) & \approx \sum_{Q\in\GG}\,\sum_{P\in \BZ:P\cap B_Q\neq \varnothing} \frac{\ell(P)\,\mu(P)}{\ell(Q)}\\
& = \sum_{P\in \BZ}\mu(P) \sum_{Q\in\GG:P\cap B_Q\neq \varnothing}\frac{\ell(P)}{\ell(Q)}.
\end{align*}
Recalling that $r(B_Q)\geq r(B_P)$ for every $Q$ in the last sum, we get
$ \sum_{Q\in\GG:P\cap B_Q\neq \varnothing}\frac{\ell(P)}{\ell(Q)}\lesssim1$, and so we deduce that
$$\int_\Gamma \sum_{Q\in\GG:x\in Q}\sum_{P\in \BZ:P\cap B_Q\neq \varnothing} \frac{\ell(P)\,\mu(P)}{\ell(Q)^{n+1}}\,d\HH^n(x) \lesssim \sum_{P\in \BZ}\mu(P)\leq \|\mu\|,$$
which implies \rf{eqsuf99*}, and completes the proof of the lemma.
\end{proof}

\vv

\subsection{The approximating measure $\sigma$ for the general case}

To prove the Main Lemma \ref{mainlemma*} in full generality,
we
consider the following Whitney decomposition of $\R^d\setminus \Gamma$: we have a family $\WW$ of dyadic cubes in $\R^d$ with disjoint interiors such that
$$\bigcup_{Q\in\WW} Q = \R^d\setminus \Gamma,$$
and moreover there are
 some constants $R>20$ and $D_0\geq1$ such the following holds for every $Q\in\WW$:
\begin{itemize}
\item[(i)] $10Q \subset \R^d\setminus \Gamma$;
\item[(ii)] $R Q \cap \Gamma \neq \varnothing$;
\item[(iii)] there are at most $D_0$ cubes $Q'\in\WW$
such that $10Q \cap 10Q' \neq \varnothing$. Further, for such cubes $Q'$, we have $\ell(Q')\approx
\ell(Q)$.
\end{itemize}
From the properties (i) and (ii) it is clear that $\dist(Q,\Gamma)\approx\ell(Q)$. We assume that
the Whitney cubes are small enough so that
\begin{equation}\label{eqeq29}
\diam(Q)< \dist(Q,\Gamma).
\end{equation}
This can be achieved by replacing each cube $Q\in\WW$ by its descendants $P\in\DD_k(Q)$, for some fixed $k\geq1$, if necessary. From \rf{eqeq29} we infer that if $Q\in\WW$ intersects some
ball $B(y,r)$ with $y\in\Gamma$, then
\begin{equation}\label{eqeq27'}
\diam(Q)\leq r,
\end{equation}
and thus 
\begin{equation}\label{eqeq28'}
Q\subset B(y,3r).
\end{equation}

For a given Borel measure $\mu$, we denote
$$M_n\mu(x) = \sup_{r>0}\frac{\mu(B(x,r))}{r^n}.$$
From the growth condition $\HH^n(\Gamma\cap B(x,r))\leq c\,r^n$ for all $x$ and $r>0$, it follows easily that the maximal operator $M_n$ is bounded
from the space of finite signed Radon measures $M(\R^d)$ into $L^{1,\infty}(\HH^n_\Gamma)$. As a consequence, for any arbitrary finite Borel measure $\mu$ in $\R^d$, $M_n\mu(x)<\infty$ for $\HH^n$-a.e.\
$x\in\Gamma.$
\vv

Denote by $\Pi_\Gamma$ the projection on $\Gamma$ given by $\Pi_\Gamma(x) = (x_1,\ldots,x_n,A(x_1,\ldots,x_n))$.
To each cube $Q\in\WW$ we associate a ball $\wt B_Q$ with radius $r(\wt B_Q)=\frac14
\ell(Q)$, centered in $\wt x_Q:=\Pi_\Gamma(x_Q)$, where $x_Q$ is the center of $Q$. 
In particular, we have $\dist(Q,\wt B_Q)\approx \ell(Q)\approx r(\wt B_Q)$.

Let $\vphi:\R^d\to\R$ be a smooth non-negative radial function supported in $B(0,1)$ which equals $1$ in $B(0,1/2)$.
For each $Q\in\WW$ we consider the function
$\vphi_Q:\R^d\to\R$ defined by
$$\vphi_Q(x) =\vphi\Bigl(\frac{x-\wt x_Q}{r(\wt B_Q)}\Bigr)$$ and then we set
\begin{equation}\label{eqgp*1}
g_Q(x) = \frac{\mu(Q)}{\|\vphi_Q\|_{L^1(\HH^n_\Gamma)}}\,\vphi_Q(x),
\end{equation}
and also
$$g(x) = \sum_{Q\in\WW} g_Q(x).$$
Notice that $\int g_Q\,d\HH^n_\Gamma = \mu(Q)$ and also that
$$\|g_Q\|_\infty\lesssim \frac{\mu(Q)}{\ell(Q)^n},\qquad \|\nabla g_Q\|_\infty\lesssim \frac{\mu(Q)}{\ell(Q)^{n+1}}.$$
Then we consider the following measure $\sigma$ on $\Gamma$, which approximates $\mu|_{\Gamma^c}$, in a sense:
\begin{equation}\label{eqsigma*1}
\sigma = g\,d\HH^n_\Gamma = \sum_{Q\in\WW}g_Q\,d\HH^n_\Gamma.
\end{equation}

Next, for some constant $A>10$ to be chosen below, we also consider the auxiliary function
$$G_A(x) = \sum_{Q\in\WW} \frac{\mu(Q)}{\ell(Q)^n}\,\chi_{A\wt B_Q}.$$
Notice that, for some absolute constant $C$, 
$$g(x) \leq C\,G_A(x)\quad\mbox{ for all $x\in\Gamma$.}$$
Also,
$$\|G_A\|_{L^1(\HH^n_\Gamma)}\leq C\,\sum_{Q\in\WW}\frac{\mu(Q)\,A^n\,r(\wt B_Q)^n}{\ell(Q)^n} \leq C\,A^n\,\|\mu\|,$$
and thus $G_A(x)<\infty$ for $\HH^n$-a.e.\ $x\in\Gamma$.

\vv

\subsection{The $\alpha_\mu$ coefficients of the $\Gamma$-cubes}

\begin{lemma}\label{lem2}
For each $Q\in\DD_\Gamma$, we have
\begin{align*}
\alpha_\mu(Q)& \lesssim \alpha_{\mu_|\Gamma,L_Q}(Q) + \alpha_{\sigma,L_Q}(Q) + \int_{6B_{Q}} \frac{\dist(x,\Gamma)}{\ell(Q)^{n+1}}\,
d\mu(x)\\
&\quad+
\inf_{x\in Q}G_A(x) \,\alpha_{\HH^n_\Gamma}(Q) + \sum_{\substack{P\in \WW:P\cap 2B_Q=\varnothing,\\
\wt B_P\cap B_Q\neq\varnothing}}
\frac{\mu(P)\,\ell(Q)}{\ell(P)^{n+1}},
\end{align*}
where $L_Q$ is the $n$-plane that minimizes $\alpha_{\HH^n_\Gamma}(Q)$.
\end{lemma}

\begin{proof}
Note first that $\alpha_\mu(Q)\leq \alpha_{\mu|_\Gamma,L_Q}(Q) + \alpha_{\mu|_{\Gamma^c,L_Q}}(Q)$.
So we just have to deal with $\alpha_{\mu|_{\Gamma^c},L_Q}(Q)$. To estimate this coefficient, we consider a $1$-Lipschitz
function $f$ supported on $B_Q$, and we write
$$\int f\,d\mu|_{\Gamma^c} = \sum_{P\in\WW}\int f\,d(\mu|_P - g_P \,\HH^n_\Gamma)
+ \int f\,d\sigma,$$
where $g_P$ is the function defined in \rf{eqgp*1} and $\sigma$  the measure in \rf{eqsigma*1}.
So, for any constant $a_Q\in\R$, we get
$$\alpha_{\mu|_{\Gamma^c},L_Q}(Q)\leq \alpha_{\sigma,L_Q}(Q) +  \frac1{\ell(Q)^{n+1}}\sup_f\biggl|\sum_{P\in\WW}\int f\,d(\mu|_P - g_P \,\HH^n_\Gamma) + a_Q\int f\,d\HH^n|_{L_Q}\biggr|,$$
where 
the supremum is taken over all $1$-Lipschitz functions supported on $B_Q$.

We denote by $I_a$ the subfamily of the cubes from $\WW$ which intersect $2B_Q$, and $I_b=\WW\setminus I_a$,
and we split the last sum above as follows:
\begin{align*}
\biggl|\sum_{P\in\WW}\int& f\,d(\mu|_P - g_P \,\HH^n_\Gamma) + a_Q\int f\,d\HH^n|_{L_Q}\biggr| \\& \leq
\sum_{P\in I_a}\biggl|\int f\,d(\mu|_P - g_P \,\HH^n_\Gamma) \biggr| + 
\biggl|\sum_{P\in I_b}\int f\,d(\mu|_P - g_P \,\HH^n_\Gamma) + a_Q\int f\,d\HH^n|_{L_Q}\biggr|\\
& =: S_a + S_b. 
\end{align*}

First we deal with the sum $S_a$. For each $P\in I_a$, since $\int g_P\,d\HH^n|_{\Gamma} = \mu(P)$, we deduce that
\begin{align*}
\left|\int f\,d(\mu|_P- g_P\,\HH^n_\Gamma)\right|& \leq 
 \left|\int_P(f(x) - f(x_P))\,d\mu(x)\right|\\
 & + 
\left|\int (f(x_P) - f(x))\, g_P(x)\,\HH^n_\Gamma(x))\right|.
\end{align*}
To deal with the first integral on the right hand side we take into account that for $x\in P$ we have
\begin{equation}\label{eqeq39}
|f(x) - f(x_P)|\leq \|\nabla f\|_\infty\,|x-x_P|\leq c\,\ell(P).
\end{equation}
Concerning the second integral, recall that $\supp g_P\subset \Gamma\cap \bar B(x_P,c\,\ell(P))$, and
thus we also have $|x-x_P|\leq c\,\ell(P)$ in the domain of integration, so that \rf{eqeq39}
holds in this case too. Therefore,
$$\left|\int f\,d(\mu|_P- g_P\,\HH^n_\Gamma)\right| \leq c\,\ell(P)\,\mu(P)\approx \int_P \dist(x,\Gamma)\,d\mu(x),$$
where we took into account that
$\dist(x,\Gamma)\approx\ell(P)$ for every $x\in P$. 
Recall that $\supp f\subset B_Q$ and since $P\in I_a$, then $P\subset 6B_Q$ by the argument in \rf{eqeq28'}. Thus,
\begin{equation}\label{eqeq41}
S_a = \sum_{P\in I_a}\left|\int f\,d(\mu|_P- g_P\,\HH^n_\Gamma)\right|
\leq  c\,\int_{6B_Q} \dist(x,\Gamma)\,d\mu(x).
\end{equation}

Next we consider the sum $S_b$. For each $P\in I_b$ we have $P\cap 2B_Q=\varnothing$, and so $\int f\,d\mu|_P=0$. Therefore,
\begin{align*}
S_b & = 
\biggl|\sum_{P\in I_b}\int f\,g_P \,d\HH^n_\Gamma - a_Q\int f\,d\HH^n|_{L_Q}\biggr|\\
& \leq
\biggl|\sum_{P\in I_b}\int f\,g_P \,d(\HH^n_\Gamma - c_{Q}\,\HH^n|_{L_Q})\biggr| + 
\int |f|\,\biggl|\sum_{P\in I_b} c_Q\,g_P -a_Q\biggr|\,d\HH^n|_{L_Q}\\
& =: S_{b,1} + S_{b,2},
\end{align*}
where $c_Q$ is the constant minimizing $\alpha_{\HH^n_\Gamma,L_Q}(Q)$.
To deal with $S_{b,1}$ we take into account that $\sum_{P\in I_b} f\,g_P$ is a Lipschitz function supported in $B_Q$, and $f\,g_P$ vanishes unless $\wt B_P\cap B_Q\neq\varnothing$.
To shorten notation, we write $P\sim Q$ if $P\in I_b$ is such that $\wt B_P\cap B_Q\neq\varnothing$. It is easy to check that $P\sim Q$ implies that $\ell(P)\gtrsim\ell(Q)$.
Then we have
\begin{align*}
\biggl\|\nabla \biggl(\sum_{P\in I_b} f\,g_P\biggr)\biggr\|_\infty &\leq \sum_{P\sim Q} \bigl(\|g_P\|_\infty + c\,\ell(Q)
\|\nabla g_P\|_\infty\bigr) \\
&\lesssim
\sum_{P\sim Q}\left(
 \frac{\mu(P)}{\ell(P)^n} + \frac{\ell(Q)\,\mu(P)}{\ell(P)^{n+1}}\right)\lesssim \sum_{P\sim Q} \frac{\mu(P)}{\ell(P)^n}.
\end{align*}
From the fact that $\ell(P)\gtrsim\ell(Q)$ it also follows that, for $A$ big enough, we have
\begin{equation}\label{eqga}
\sum_{P\sim Q} \frac{\mu(P)}{\ell(P)^n}\lesssim \inf_{x\in Q}G_A(x).
\end{equation}
Therefore,
$$S_{b,1}\lesssim \alpha_{\HH^n_\Gamma}(Q)\,\ell(Q)^{n+1}\,\inf_{x\in Q}G_A(x).$$

Finally we turn our attention to the term $S_{b,2}$. Choosing $a_Q = c_Q \sum_{P\sim Q}g_P(x_Q)$ and using that $|c_Q|\lesssim1$, we obtain
$$S_{b,2}\lesssim 
\int \biggl|f\sum_{P\sim Q} \bigl[g_P - g_P(x_Q)\bigr]\biggr|\,d\HH^n|_{L_Q}.
$$
Note now that, for $x\in B_Q$, 
$$\sum_{P\sim Q} \bigl|g_P(x) - g_P(x_Q)\bigr|
\lesssim \sum_{P\sim Q} \|\nabla g_P\|_\infty\,\ell(Q)\lesssim  \sum_{P\sim Q}
\frac{\mu(P)\,\ell(Q)}{\ell(P)^{n+1}},$$
and so, using also that $|f|\leq C\,\ell(Q)\,\chi_{B_Q}$,
$$S_{b,2}\lesssim \sum_{P\sim Q}
\frac{\mu(P)\,\ell(Q)^{n+2}}{\ell(P)^{n+1}}.$$

Gathering all the estimates above, we obtain
$$\alpha_{\mu|_{\Gamma^c},L_Q}(Q)\lesssim\alpha_{\sigma,L_Q}(Q) + \int_{6B_{Q}} \frac{\dist(x,\Gamma)}{\ell(Q)^{n+1}}\,
d\mu(x) + \inf_{x\in Q}G_A(x) \,\alpha_{\HH^n_\Gamma}(Q) + \sum_{P\sim Q}
\frac{\mu(P)\,\ell(Q)}{\ell(P)^{n+1}},$$
which completes the proof of the lemma.
\end{proof}

\vv

\subsection{Proof of Main Lemma \ref{mainlemma*}}

It is immediate to check that the statement in Main Lemma \ref{mainlemma*} is equivalent to the fact that
\begin{equation}\label{eqalp1}
\sum_{Q\in\DD_\Gamma}\alpha_\mu(Q)^2\,\chi_Q(x)<\infty \quad \mbox{for $\HH^n$-a.e.\ $x\in\Gamma$.}
\end{equation}
To prove this estimate we intend to use Lemma \ref{lem2}.
To this end, observe that, by Lemma \ref{lem1} we know that
$$\sum_{Q\in\DD_\Gamma} \bigl(\alpha_{\mu|_\Gamma,L_Q}(Q)^2 + \alpha_{\sigma,L_Q}(Q)^2\bigr)\,\chi_Q(x)<\infty \quad \mbox{for $\HH^n$-a.e.\ $x\in\Gamma$.}$$
Also, for every $x\in\Gamma$,
$$\sum_{Q\in\DD_\Gamma:Q\ni x}\inf_{y\in Q}G_A(y)^2 \,\alpha_{\HH^n_\Gamma}(Q)^2 
\leq G_A(x)^2\,\sum_{Q\in\DD_\Gamma:Q\ni x}\alpha_{\HH^n_\Gamma}(Q)^2 <\infty \quad \mbox{for $\HH^n$-a.e.\ $x\in\Gamma$,}$$ 
since $G_A(x)<\infty$ and $\sum_{Q\in\DD_\Gamma:Q\ni x}\alpha_{\HH^n_\Gamma}(Q)^2 <\infty$ for $\HH^n$-a.e.\ $x\in\Gamma$.

Next we show that 
\begin{equation}\label{eqalp2}
\sum_{Q\in\DD_\Gamma}\biggl(\,\sum_{\substack{P\in \WW:P\cap 2B_Q=\varnothing,\\
\wt B_P\cap B_Q\neq\varnothing}}
\frac{\mu(P)\,\ell(Q)}{\ell(P)^{n+1}}\biggr)^2\,\chi_Q(x)<\infty \quad \mbox{for $\HH^n$-a.e.\ $x\in\Gamma$.}
\end{equation}
To this end, note that, as in \rf{eqga},
$$\sum_{\substack{P\in \WW:P\cap 2B_Q=\varnothing,\\
\wt B_P\cap B_Q\neq\varnothing}}
\frac{\mu(P)\,\ell(Q)}{\ell(P)^{n+1}}\,\chi_Q(x)
\lesssim 
\sum_{\substack{P\in \WW:P\cap 2B_Q=\varnothing,\\
\wt B_P\cap B_Q\neq\varnothing}}
\frac{\mu(P)}{\ell(P)^{n}}\,\chi_Q(x) \lesssim G_A(x).$$
Hence, 
$$\sum_{Q\in\DD_\Gamma}\biggl(\sum_{\substack{P\in \WW:P\cap 2B_Q=\varnothing,\\
\wt B_P\cap B_Q\neq\varnothing}}
\frac{\mu(P)\,\ell(Q)}{\ell(P)^{n+1}}\biggr)^2\,\chi_Q(x) \lesssim 
G_A(x)
\sum_{Q\in\DD_\Gamma}\sum_{\substack{P\in \WW:P\cap 2B_Q=\varnothing,\\
\wt B_P\cap B_Q\neq\varnothing}}
\frac{\mu(P)\,\ell(Q)}{\ell(P)^{n+1}}\chi_Q(x).
$$
Thus to prove \rf{eqalp2} it suffices to show that 
\begin{equation}\label{eqalp3}
\sum_{Q\in\DD_\Gamma}\sum_{\substack{P\in \WW:P\cap 2B_Q=\varnothing,\\
\wt B_P\cap B_Q\neq\varnothing}}
\frac{\mu(P)\,\ell(Q)}{\ell(P)^{n+1}}\chi_Q(x)<\infty \quad \mbox{for $\HH^n$-a.e.\ $x\in\Gamma$.}
\end{equation}
By Fubini we have
\begin{align*}
\int \sum_{Q\in\DD_\Gamma}\sum_{\substack{P\in \WW:P\cap 2B_Q=\varnothing,\\
\wt B_P\cap B_Q\neq\varnothing}}
\frac{\mu(P)\,\ell(Q)}{\ell(P)^{n+1}}\,&\chi_Q(x)\,d\HH^n_\Gamma(x)  \approx
\sum_{Q\in\DD_\Gamma}\sum_{\substack{P\in \WW:P\cap 2B_Q=\varnothing,\\
\wt B_P\cap B_Q\neq\varnothing}}
\frac{\mu(P)\,\ell(Q)}{\ell(P)^{n+1}}\,\ell(Q)^n\\
& \approx
\sum_{\substack{P\in \WW}}\mu(P)
\sum_{\substack{Q\in\DD_\Gamma:P\cap 2B_Q=\varnothing,\\
\wt B_P\cap B_Q\neq\varnothing}} 
\frac{\ell(Q)^{n+1}}{\ell(P)^{n+1}} \lesssim \sum_{\substack{P\in \WW}}\mu(P) = \|\mu\|,
\end{align*}
where in the last inequality we took into account that the $\Gamma$-cubes $Q\in\DD_\Gamma$ such that $P\cap 2B_Q=\varnothing$
and $\wt B_P\cap B_Q\neq\varnothing$ satisfy $Q\subset c\wt B_P$ for some $c>1$, because $\ell(Q)\lesssim\ell(P)$. So \rf{eqalp3}
is proved.

Finally, to complete the proof of this lemma, by Lemma \ref{lem2} it remains to show that 
\begin{equation}\label{eqalp5}
\sum_{Q\in\DD_\Gamma}\left(\int_{6B_{Q}} \frac{\dist(y,\Gamma)}{\ell(Q)^{n+1}}\,
d\mu(y)\right)^2\,\chi_Q(x)<\infty \quad \mbox{for $\HH^n$-a.e.\ $x\in\Gamma$.}
\end{equation}
To this end, by Cauchy-Schwarz we get
$$\left(\int_{6B_{Q}} \frac{\dist(y,\Gamma)}{\ell(Q)^{n+1}}\,
d\mu(y)\right)^2\leq \mu(6B_Q)
\int_{6B_{Q}} \left(\frac{\dist(y,\Gamma)}{\ell(Q)^{n+1}}\right)^2\,
d\mu(y).$$
Since 
$\frac{\mu(6B_Q)}{\ell(Q)^n}\lesssim M_n \mu(x)$ for all $x\in Q$, and $M_n\mu(x)<\infty$ for $\HH^n$-a.e.\ $x\in\Gamma$, 
it turns that, to prove \rf{eqalp5}, it suffices to show that
\begin{equation}\label{eqalp6}
\sum_{Q\in\DD_\Gamma}\int_{6B_{Q}} \frac{\dist(y,\Gamma)^2}{\ell(Q)^{n+2}}\,
d\mu(y)\,\chi_Q(x)<\infty \quad \mbox{for $\HH^n$-a.e.\ $x\in\Gamma$.}
\end{equation}
The integral of the left hand side of \rf{eqalp6} with respect to $\HH^n_\Gamma$ does not exceed
\begin{align*}
c\sum_{Q\in\DD_\Gamma} \int_{6B_{Q}} \frac{\dist(y,\Gamma)^2}{\ell(Q)^{n+2}}\,
d\mu(y)\,\ell(Q)^n = c\sum_{Q\in\DD_\Gamma} \int_{6B_{Q}} \frac{\dist(y,\Gamma)^2}{\ell(Q)^2}\,
d\mu(y).
\end{align*}
By Fubini, this equals
$$ c\int\dist(y,\Gamma)^2\sum_{Q\in\DD_\Gamma}\chi_{6B_{Q}}(y)\, \frac1{\ell(Q)^{2}} \,
d\mu(y).$$
Notice now that
$$\sum_{Q\in\DD_\Gamma}\chi_{6B_{Q}}(y)\, \frac1{\ell(Q)^2}
= \sum_{Q\in\DD_\Gamma: y\in 6B_{Q}} \frac1{\ell(Q)^2} \lesssim \frac1{\dist(y,
\Gamma)^2},$$
because the condition $y\in 6B_{Q}$ implies that $\dist(y,\Gamma)\leq r(B_Q)\approx\ell(Q)$.
Thus,
\begin{align*}
\sum_{Q\in\DD_\Gamma} 
\int_{6B_{Q}} \frac{\dist(y,\Gamma)^2}{\ell(Q)^{2}}\,
d\mu(y) & \lesssim
\int \frac{\dist(y,\Gamma)^2}{\dist(y,\Gamma)^2} \,
d\mu(y)= \|\mu\|.
\end{align*}
Hence, \rf{eqalp5} follows and we are done.
\fiproof

\vvv


\section{The proof of Theorem \ref{teo2}}\label{sec2}

\subsection{Peliminaries}

The case $p=1$ of Theorem \ref{teo2} follows from the fact that
\begin{equation}\label{eqcl331}
\beta_{\mu,1}(x,r)\leq c\,\alpha_\mu(x,2r) \quad \mbox{for all $x\in\supp\mu$, $r>0$.}
\end{equation}
To see this, take an $n$-plane $L\subset\R^d$ and $a\geq0$ which minimize $\alpha_\mu(x,2r)$, let $\vphi$ be a Lipschitz function
supported on $\bar B(x,2r)$ which equals $1$ on $\bar B(x,r)$, with ${\rm Lip}(\vphi)\leq 1/r$. 
 Then
\begin{align*}
\int_{\bar B(x,r)}\dist(y,L)\,d\mu(y) & \leq \int_{\bar B(x,r)}\vphi(y)\,\dist(y,L)\,d\mu(y)\\
& = \left|\int \vphi(y)\,\dist(y,L)\,d(\mu - a \HH^n_L)(y)\right|\\
& \leq  {\rm Lip}\bigl(\vphi\,\dist(\cdot,L)\bigr)\,\dist_{2B}(\mu , a \HH^n_L)
\\ & \leq c\,r^{n+1}\,\alpha_\mu(x,2r),
\end{align*}
which yields \rf{eqcl331}.

Notice also that, for $1\leq p<2$, given a ball $B(x,r)$ and any $n$-plane $L$, 
by H\"older's inequality we have
\begin{multline*}
\frac1{r^n} \int_{\bar B(x,r)} \left(\frac{\dist(y,L)}{r}\right)^p\,d\mu(y)\\
\leq   \left(\frac1{r^n}\int_{\bar B(x,r)} \left(\frac{\dist(y,L)}{r}\right)^2\,d\mu(y)\right)^{p/2} 
\left(\frac{\mu(\bar B(x,r))}{r^n}\right)^{1-p/2}.
\end{multline*}
So taking infimums and raising to the power $1/p$, we obtain
$$\beta_{\mu,p}(x,r)\leq \left(\frac{\mu(\bar B(x,r))}{r^n}\right)^{\frac1p-\frac12}\,\beta_{\mu,2}(x,r).$$
As a consequence, for all $x\in\R^d$,
$$\int_0^\infty \beta_{\mu,p}(x,r)^2\,\frac{dr}r \leq \left(\sup_{r>0}
\frac{\mu(\bar B(x,r))}{r^n}\right)^{\frac2p-1} \int_0^\infty \beta_{\mu,2}(x,r)^2\,\frac{dr}r.$$
If $\mu$ is a finite Borel measure which is rectifiable, then the supremum on the right hand side above
is finite for $\mu$-a.e.\ $x\in\R^d$. So to prove Theorem \ref{teo2} it suffices to show that
\begin{equation}\label{eqteo2}
\int_0^\infty \beta_{\mu,2}(x,r)^2\,\frac{dr}r<\infty \quad\mbox{ for $\mu$-a.e.\ $x\in\R^d$.}
\end{equation}
To prove this statement we will follow an argument inspired by some techniques from \cite[Lemma 5.2]{Tolsa-plms},
where it is shown that the $\beta_{\mu,2}$'s can be estimated in terms of the $\alpha_\mu$ coefficients
when $\mu$ is an $n$-dimensional AD-regular measure. In the present situation, $\mu$ fails to be AD-regular
(in general) and so we will need to adapt the techniques in \cite{Tolsa-plms} by suitable stopping time arguments.


\subsection{The stopping cubes}

We denote by $\DD$ the family of dyadic cubes from $\R^d$. Also, given $R\in\DD$, $\DD(R)$ stands for
the cubes from $\DD$ which are contained in $R$. 

 Since $\mu$ is $n$-rectifiable, the density
$$\Theta^{n}(x,\mu)= \lim_{r\to 0}\frac{\mu(B(x,r))}{(2r)^n}$$
exists and is positive $\mu$-almost everywhere. So, given $R\in\DD$ with $\mu(R)>0$ and $\ve>0$, there exists $N>0$ big enough
such that 
$$\mu\left(\{x\in R:\,N^{-1}\leq\Theta^n(x,\mu)\leq N\}\right) > (1-\ve)\,\mu(R).$$
Let $r_0>0$ and denote now
$$A = A(N,r_0) = \{x\in R: \,N^{-1}\,r^n\leq \mu(B(x,r))\leq 4N\,r^n\mbox{ for $0<r\leq r_0$}\}.$$
Then we infer that 
$$\mu(R\setminus A)\leq 2\ve$$
if $r_0$ is small enough.

By Theorem \ref{teo1} we know that
$$\int_0^\infty \alpha_\mu(x,r)^2\,\frac{dr}r<\infty.$$
So setting
$$F = F(N) = \left\{x\in R\cap\supp\mu: \,\ds\int_0^\infty \alpha_\mu(x,r)^2\,\frac{dr}r\leq N\right\},$$
it turns out that
$$\mu(R\setminus F)\leq \ve\,\mu(R)$$
if $N$ is big enough.

We take $N$ and $r_0$ so that
\begin{equation}\label{eqjk21}
\mu(R\setminus (A\cap F))\leq
\mu(R\setminus A) + \mu(R\setminus F) \leq 3\ve\,\mu(R).
\end{equation}

For a given cube $Q\in\DD$, we denote
$B_Q = \bar B(x_Q,3\diam(Q))$, where $x_Q$ stands for the center of $Q$.
Given some big constant $M>N$, we consider now the following subfamilies of cubes from $\DD(R)$:
\begin{itemize}
\item We say that $Q\in\DD$ belongs to $\HD_0$ if $Q\subset 3R$, $\diam(Q)\leq r_0/10$ and
$\mu(B_Q)\geq M\,\ell(Q)^n$.

\item We say that $Q\in\DD$ belongs to $\LD_0$ if $Q\subset 3R$, $\diam(Q)\leq r_0/10$ and
$\mu(3Q)\leq  M^{-1}\,\ell(Q)^n$.

\item We say that $Q\in\DD$ belongs to $\BA_0$ if $Q\subset 3R$, $\diam(Q)\leq r_0/10$, 
$Q\not\in \HD_0\cup \LD_0$, and
$Q\cap F= \varnothing$.

\end{itemize}

We denote by $\sss$ the family of maximal (and thus disjoint) cubes from $\HD_0\cup\LD_0\cup\BA_0$.
We set $\HD=\sss\cap\HD_0$, $\LD=\sss\cap\LD_0$, and $\BA=\sss\cap\BA_0$. The notations 
$\HD$, $\LD$, and $\BA$ stand for ``high density'', ``low density'', and ``big alpha's'', respectively.
\vv

\begin{lemma}\label{lemstopetit}
For $M$ big enough, we have
$$R\cap \bigcup_{Q\in \sss} Q \subset (R\setminus A)\cup (R\setminus F),$$
and thus
$$\mu\biggl(R\cap \bigcup_{Q\in \sss} Q \biggr)\leq 3\ve\,\mu(R).$$
\end{lemma}

\begin{proof}
Since the second statement is an immediate consequence of the first one, we only have to show that
if $Q\in\DD(R)\cap\sss$, then $Q\subset (R\setminus A)\cup (R\setminus F)$.

Suppose first that $Q\in\HD$. Since for any $x\in Q$ we have $B_{Q}\subset B(x,6\diam(Q))$, setting
$r=6\,\diam(Q)$ we get
$$\mu(B(x,r)) \geq \mu(B_Q) \geq M\,\ell(Q)^n = c_7\,M\,r^n> 4N\,r^n,$$
assuming $M> c_7^{-1}4N$.
Since $r=\diam(6Q)\leq 6r_0/10\leq r_0$, it turns out that $x\not\in A$. Hence $Q\subset R\setminus A$.

Consider now a cube $Q\in\LD$. Notice that $B(x,\ell(Q))\subset 3Q$ for every $x\in Q$. Thus,
$$\mu\bigl(B(x,\ell(Q))\bigr)\leq \mu(3Q)\leq \frac1M \,\ell(Q)^n.$$
Thus, $x\in R\setminus A$ because $M>N$. So $Q\subset R\setminus A$.

Finally, if $Q\in \BA$, then $Q\cap F=\varnothing$ and thus $Q\subset R\setminus F$.
\end{proof}

\vv

We denote by $\mathcal G$ the subset of the cubes from $\DD$ with $\diam(Q)\leq r_0/10$ which are not
contained in any cube from $\sss$. We also set $\GZ(R) = \GZ \cap \DD(R)$.

For a given cube $Q\in\DD$, we denote
\begin{equation}\label{eq000}
\alpha_\mu(Q) = \alpha_\mu(B_Q).
\end{equation}
Recall that $B_Q = \bar B(x_Q,3\diam(Q))$.
\vv

\begin{lemma}\label{lemfac55}
For all $x\in 3R\cap\supp\mu$, we have
$$\sum_{Q\in\GZ:x\in Q} \alpha_\mu(Q)^2 \leq c\,N.$$
\end{lemma}

\begin{proof}
Let $Q\in\GZ$ and $z\in Q\cap\supp\mu$. Since $B_Q\subset \bar B(z,6\diam(Q))$, for any $r\in
[6\,\diam(Q),12\,\diam(Q)]$ we have 
$$\alpha_\mu(Q) \leq c\,\alpha_\mu(z,r),$$
and thus
\begin{equation}\label{eqg332}
\alpha_\mu(Q)^2\leq c\int_{6\diam(Q)}^{12\diam(Q)}\alpha_\mu(z,r)^2\,\frac{dr}r.
\end{equation}

Given $x\in 3R\cap\supp\mu$, consider some cube $P\in\GZ$ such that $x\in P$.
Since $P\not\in \BA$, there exists some $z\in F\cap P$, and then
from \rf{eqg332} we derive
\begin{align*}
\sum_{Q\in\GZ:Q\supset P} \alpha_\mu(Q)^2 & \leq c\sum_{Q\in\GZ: Q\supset P} \int_{6\diam(P)}^{12\diam(P)}\alpha_\mu(z,r)^2\,\frac{dr}r \\
&\leq  c \int_0^\infty\alpha_\mu(z,r)^2\,\frac{dr}r\leq c\,N.
\end{align*}
Since this holds for all $P\in\GZ$ which contains $x$, the lemma follows.
\end{proof}

\vv


\subsection{A key estimate}

\begin{lemma}\label{lemakeyyy}
Let $Q\in \GZ(R)$. Let $L_Q$ be the line minimizing $\alpha(Q)$ and $x\in 3Q\cap\supp\mu$. 
If there exists some $S_x\in\sss$ such that $x\in S_x$, then set $\ell_x=\ell(S_x)$. Otherwise, set $\ell_x=0$.
We have
$$\dist(x,L_Q)\leq c(M)\sum_{P\in\GZ:x\in P\subset 3Q}\alpha_\mu(P)\,\ell(P) + c\,\ell_x.$$
\end{lemma}

We will not prove this result in detail because the arguments are almost the same as the ones in 
Lemma 5.2 of \cite{Tolsa-plms}. We just give a concise sketch. 

\begin{proof}[Sketch of the proof]
Let $x\in 3Q\cap\supp\mu$ and suppose that $\ell_x\neq 0$. For $i\geq1$, denote by $Q_i$ the dyadic cube with side length $2^{-i}\ell(Q)$ that contains $x$, so that $Q_m$ is the parent of the cube $S_x$ in the lemma, and $Q_i\in\GZ(R)$ for $1\leq i \leq m$. Set also $Q_0=Q$. For $0\leq i \leq m$,
let $L_{Q_i}$ be some $n$-plane minimizing $\alpha_\mu(Q_i)$ and denote by $\Pi_i$ the orthogonal
projection onto $L_{Q_i}$. 

Let $x_m=\Pi_m(x)$, and by backward induction set $x_{i-1}=\Pi_{i-1}(x_i)$ for $i=m,\ldots,1$.
Then we set
\begin{equation}\label{eq001}
\dist(x,L_Q)\leq |x_0-x|\leq \sum_{i=1}^{m}|x_{i-1} - x_i| + |x_{m}-x|.
\end{equation}
It is clear that $|x_{m-1}-x|\lesssim \ell_x$, and one can check also that, for $1\leq i \leq m$,
\begin{equation}\label{eq002}
|x_{i-1} - x_i|\lesssim \dist_H(L_{Q_{i-1}}\cap B_{Q_i}, \,L_{Q_i}\cap B_{Q_i}),
\end{equation}
where $\dist_H$ stands for the Hausdorff distance. Further, it turns out that
\begin{equation}\label{eq003}
\dist_H(L_{Q_{i-1}}\cap B_{Q_i}, \,L_{Q_i}\cap B_{Q_i})\lesssim \alpha_\mu(Q_i)\,\ell(Q_i),
\end{equation}
with the implicit constant depending on $M$. This estimate has been proved in Lemma 3.4 of
\cite{Tolsa-plms} in the case when $\mu$ is AD-regular. It is not difficult to check that
the same arguments also work for the cubes $Q_i$, $1\leq i \leq m$, due to the fact that
$$M^{-1}\ell(Q_i)^n\leq \mu(3Q_i)\leq \mu(B_{Q_i})\leq M\,\ell(Q_i)^n.$$
From \rf{eq001}, \rf{eq002} and \rf{eq003}, the lemma follows.
\end{proof}

\vv


\subsection{Proof of \rf{eqteo2}}

Given a cube $Q \subset \R^d$, we set
\begin{equation}\label{eq004}
\beta_{\mu,2}(Q) = \inf_L \left(\frac1{\ell(Q)^n} \int_{3Q} \left(\frac{\dist(y,L)}{\ell(Q)}\right)^2\,d\mu(y)\right)^{1/2},
\end{equation}
where the infimum is taken over all $n$-planes $L\subset \R^d$. Instead, we could also have set
$\beta_{\mu,2}(Q)=\beta_{\mu,2}(B_Q)$, analogously to the definition of $\alpha_\mu(Q)$ in \rf{eq000}. However,
for technical reasons, the definition in \rf{eq004} is more appropriate.

To prove \rf{eqteo2} we will show first the next result.

\begin{lemma}\label{lemeno}
The following holds:
$$\sum_{Q\in\GZ(R)} \beta_{\mu,2}(Q)^2\,\mu(Q)\leq C(M,N)\,\mu(3R).$$
\end{lemma}

\begin{proof}
Consider a cube $Q\in\GZ(R)$. By Lemma \ref{lemakeyyy}, for all $x\in 3Q\cap\supp\mu$ we have
$$\dist(x,L_Q)\leq c(M)\sum_{P\in\GZ:x\in P\subset 3Q}\alpha_\mu(P)\,\ell(P) + c\,\ell_x.$$
So we get
\begin{align*}
\dist(x,L_Q)^2 & \leq
c(M)\left(\sum_{P\in\GZ:x\in P\subset 3Q}\alpha_\mu(P)\,\ell(P)\right)^2 + c\,\ell_x^2\\
& \leq c(M)\sum_{P\in\GZ:x\in P\subset 3Q}\alpha_\mu(P)^2\,\ell(P)\ell(Q) + c\,\ell_x^2.
\end{align*}
Then we have
\begin{align*}
\beta_{\mu,2}(Q)^2 & \lesssim_M 
\frac1{\ell(Q)^{n+2}} \int_{3Q} 
\sum_{P\in\GZ:P\subset 3Q}\alpha_\mu(P)^2\,\ell(P)\ell(Q)\chi_P(x) \,d\mu(x)\\
& \quad + \frac1{\ell(Q)^{n+2}} \int_{3Q} \sum_{P\in\sss:P\subset 3Q} \ell(P)^2\chi_P(x)
\,d\mu(x)\\
& = \sum_{P\in\GZ:P\subset 3Q}\alpha_\mu(P)^2\,\frac{\mu(P)\ell(P)}{\ell(Q)^{n+1}} +
 \sum_{P\in\sss:P\subset 3Q} \frac{\mu(P)\ell(P)^2}{\ell(Q)^{n+2}}.
\end{align*}
Thus we obtain
\begin{align}\label{eqIiII}
\sum_{Q\in\GZ(R)} \beta_2(Q)^2\,\mu(Q) & \lesssim_M 
\sum_{Q\in\GZ(R)}\sum_{P\in\GZ:P\subset 3Q}\alpha_\mu(P)^2\,\frac{\mu(P)\,\ell(P)\,\mu(Q)}{\ell(Q)^{n+1}}
\nonumber\\
\quad &+\sum_{Q\in\GZ(R)}
\sum_{P\in\sss:P\subset 3Q} \frac{\mu(P)\ell(P)^2\,\mu(Q)}{\ell(Q)^{n+2}} = : I+ II.
\end{align}

First we deal with the term $II$. By Fubini, we have
$$II= \sum_{P\in\sss}\mu(P)\,\ell(P)^2
\sum_{Q\in\GZ(R):3Q\supset P} \frac{\mu(Q)}{\ell(Q)^{n+2}}.$$
Since $\mu(Q)\leq M\,\ell(Q)^n$ for all $Q\in\GZ(R)$, the last sum above does not exceed $C(M)/\ell(P)^2$.
Thus,
$$II\leq C(M) \sum_{P\in\sss}\mu(P)\leq C(M)\,\mu(3R).$$

Finally, we turn our attention to the term $I$ in \rf{eqIiII}:
$$I= \sum_{P\in\GZ:P\subset 3R}\alpha_\mu(P)^2\,\mu(P)\,\ell(P)
\sum_{Q\in\GZ(R):3Q\supset P} \frac{\mu(Q)}{\ell(Q)^{n+1}}.$$
Using again that $\mu(Q)\leq M\,\ell(Q)^n$ for all $Q\in\GZ(R)$, we derive
$$I \leq c(M) \sum_{P\in\GZ:P\subset 3R}\alpha_\mu(P)^2\,\mu(P).$$
By Lemma \ref{lemfac55}, the sum on the right hand side above does
not exceed $C(N)\,\mu(3R)$, and so the lemma follows.
\end{proof}
\vv

Now we can easily prove the estimate \rf{eqteo2}. Indeed,  we have
$$
\int_{A\cap F}\int_0^{c_8\,r_0} \beta_{\mu,2}(x,r)^2\,\frac{dr}r\leq c
\sum_{Q\in\GZ(R)} \beta_{\mu,2}(Q)^2\,\mu(Q)\leq C(M,N)\,\mu(3R).$$
Thus
$$\int_0^\infty \beta_{\mu,2}(x,r)^2\,\frac{dr}r<\infty \quad\mbox{ for $\mu$-a.e.\ $x\in A\cap F$.}$$
Recalling that, by \rf{eqjk21}, $\mu(R\setminus (A\cap F)) \leq 3\ve\,\mu(R)$ and that $\ve$ can 
be taken arbitrarily small, it turns out that 
$$\int_0^\infty \beta_{\mu,2}(x,r)^2\,\frac{dr}r<\infty \quad\mbox{ for $\mu$-a.e.\ $x\in R$.}$$
Since this holds for any dyadic cube $R$ with $\mu(R)>0$, \rf{eqteo2} is proved.


\vvv

\end{document}